\documentclass[11pt]{article}

\usepackage[utf8]{inputenc}  
\usepackage[english]{babel}
\usepackage{amsmath,amssymb,amsfonts,amsthm}
\usepackage{xcolor}
\usepackage{enumerate}
\usepackage{tikz,tkz-tab}
\usepackage{array}
\usepackage{hyperref}
\usepackage[mono=false]{libertine}
\useosf
\usepackage{mathpazo}

\usepackage{wasysym}

\usepackage[a4paper,vmargin={3.5cm,3.5cm},hmargin={2.5cm,2.5cm}]{geometry}
\usepackage[font=sf, labelfont={sf,bf}, margin=1cm]{caption}
\usepackage{graphicx,graphics}
\usepackage{epsfig}
\usepackage{latexsym}
\usepackage{graphicx}

\theoremstyle{plain}
\newtheorem{theorem}{Theorem}
\newtheorem{lemma}[theorem]{Lemma}
\newtheorem{proposition}[theorem]{Proposition}

\theoremstyle{definition}

\theoremstyle{remark}

\newcommand{\prob}[1]{\mathbb{P}\left(#1\right)}
\newcommand{\eqd}{\stackrel{(d)}{=}}

\newcommand{\RR}{\mathbb{R}}

\newcommand{\cont}[1]{#1}

\title{Iterated Brownian motion ad libitum is not the pseudo-arc}
\author{J\'er\^{o}me Casse\thanks{Université Paris-Dauphine. Email: \texttt{jerome.casse.math@gmail.com}}\ \ and Nicolas Curien\thanks{Université Paris-Saclay and Institut Universitaire de France.  E-mail: \texttt{nicolas.curien@gmail.com}}}

\date{}
\begin{document}
\maketitle 
\begin{abstract}
  We show that the construction of a random continuum $\mathcal{C}$ from independent two-sided Brownian motions as considered in \cite{KS20} almost surely yields a non-degenerate indecomposable but not-hereditary indecomposable continuum. In particular $\mathcal{C}$ is (unfortunately) not the pseudo-arc.
\end{abstract}

\section{Introduction}
\paragraph{Iterated Brownian motions ad libitum.}
Let $( \mathfrak{B}_i)_{i \geq 1}$ be a sequence of i.i.d.\ two-sided Brownian motions (BM), i.e.\ $( \mathfrak{B}_i(t))_{t \geq 0}$ and $( \mathfrak{B}_i(-t))_{t \geq 0})$ are independent standard linear Brownian motions started from $0$. The $n$th iterated BM is
\begin{equation}
  I^{(n)} =  \mathfrak{B}_1 \circ \dots \circ  \mathfrak{B}_{n}.
\end{equation}
The doubly iterated Brownian motion $I^{(2)}$ has been deeply studied in the 90's. It permits to construct solutions to partial differential equations~\cite{Funaki79} and lots of results about its probabilistic and analytic properties can be found in~\cite{Bertoin96,Burdzy93,BK95,ES99,KL99,OB09,Xiao98} and references therein. Of course $ I^{(n)}$ is wilder and wilder as $n$ increases (see Figure \ref{fig:iteration}) but in \cite{CK14}, second author and Konstantopoulos proved that \emph{the occupation measure} of $I^{(n)}$ over $[0,1]$ converges as $n \to \infty$ towards a random probability measure $\Xi$ which can be though of as iterated Brownian motions ad libitum. This object has then been studied in~\cite{CM16} by the first author and Marckert, and they gave a description of $\Xi$ using invariant measure of an iterated functions system (IFS). However, many distributional properties of $\Xi$ remain open.
 
\begin{figure}[!h]
  \begin{center}
    \includegraphics[height=3cm]{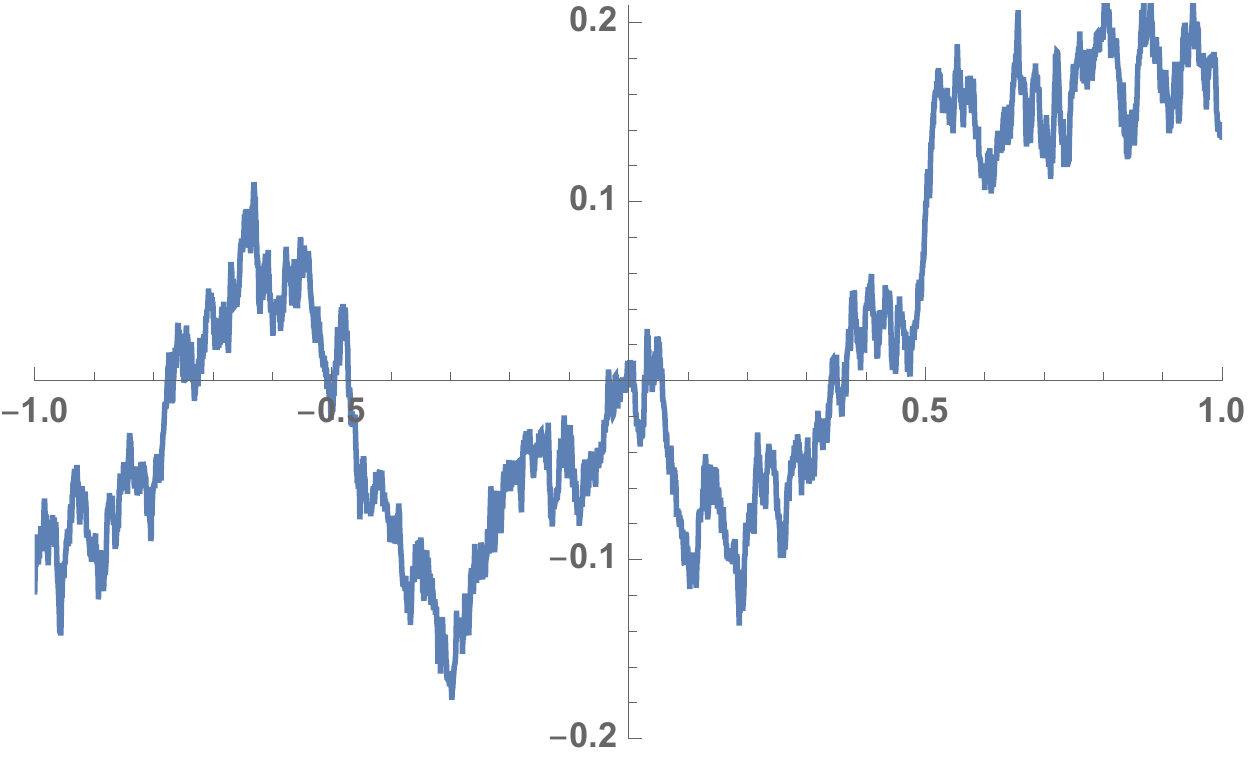}
    \includegraphics[height=3cm]{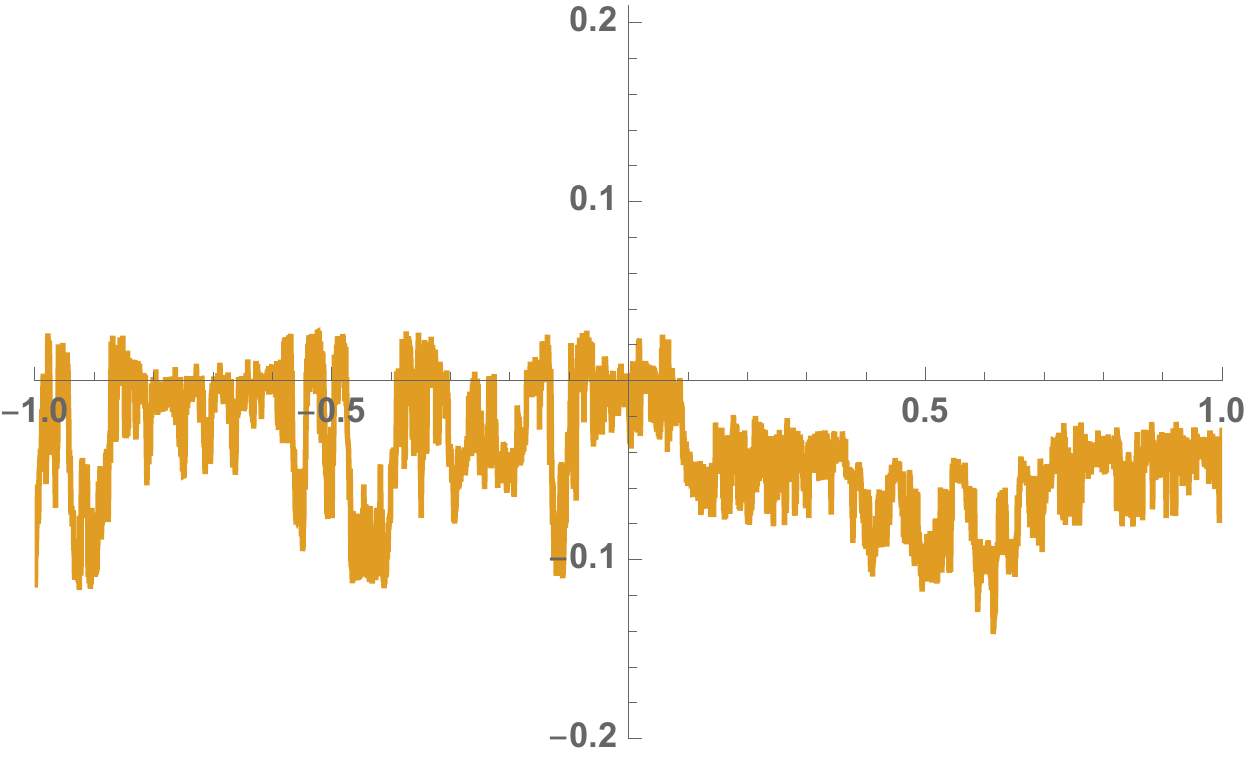}
    \includegraphics[height=3cm]{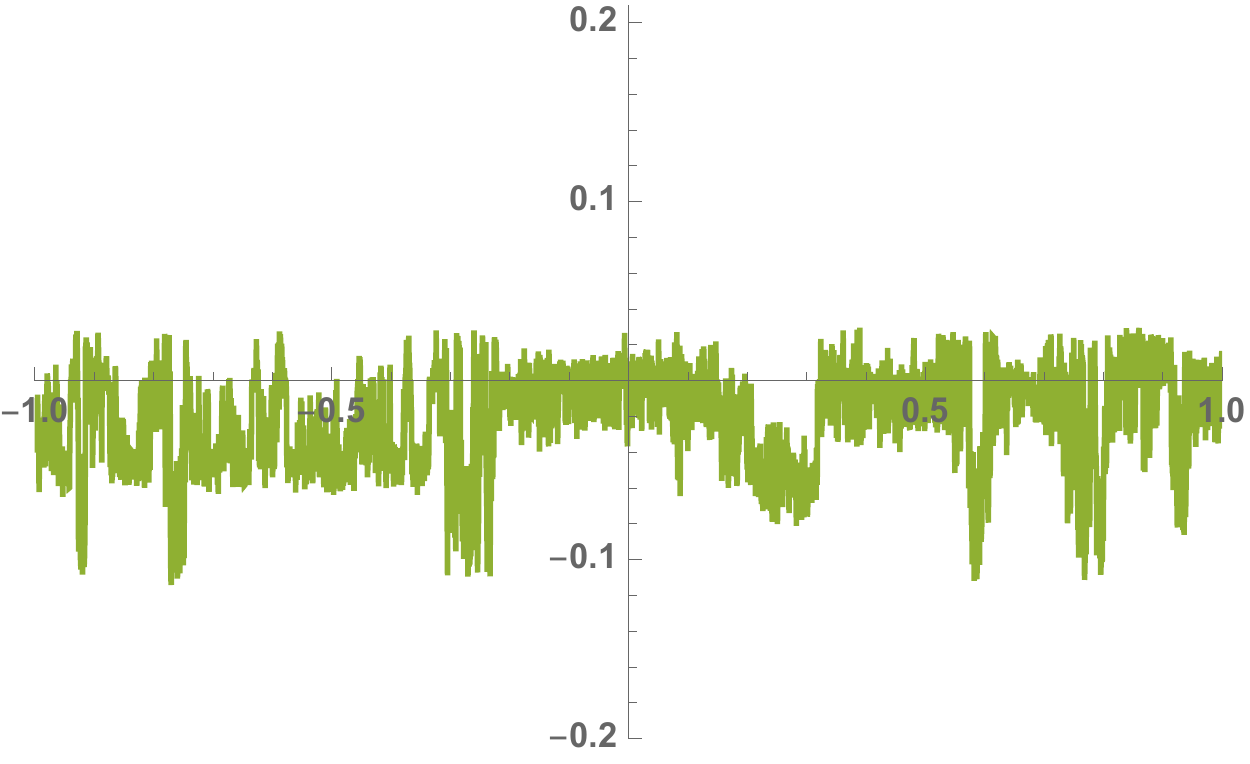}
    \caption{Simulations of $I^{{(1)}},I^{{(2)}}$ and $I^{{(3)}}$, the first three iteration of independent two-sided Brownian motions. The article studies random continuum build out the sequence of $(I^{{(n)}} : n \geq 1)$. \label{fig:iteration}}
  \end{center}
\end{figure}

\paragraph{Continuum and pseudo-arc.} In a recent work, Kiss and Solecki used iterated Brownian motions to define a random \emph{continuum}. Recall that a \emph{continuum} is a nonempty, compact, connected metric space. They were interested by the so-called \emph{pseudo-arc}. The \emph{pseudo-arc} is a homogeneous continuum which is similar to an arc, so similar, that its existence was unclear in the beginning of the last century. A continuum $\cont{C}$ is 
\begin{itemize}
\item \emph{chainable} (also called \emph{arc-like}, see~\cite[Theorem~12.11]{Nadler92}), if for each $\varepsilon > 0$, there exists a continuous function $f: C \to [0,1]$ such that the pre-images of points under $f$ have diameter less than $\varepsilon$. 
\item \emph{decomposable}, if there exist $\cont{A}$ and $\cont{B}$ two subcontinua of $\cont{C}$ such that $\cont{A},\cont{B} \neq \cont{C}$ and $\cont{C} = \cont{A} \cup \cont{B}$. A non decomposable continuum is called \emph{indecomposable}. 
\item \emph{hereditarily indecomposable} if any of its subcontinuum (non reduced to a singleton) is indecomposable.
\end{itemize}
By~\cite{Bing51}, the \emph{pseudo-arc} is the unique (up to homeomorphisms) chainable and hereditarily indecomposable continuum non reduced to a singleton. In particular, any subcontinuum (non reduced to a singleton) of a pseudo-arc is a pseudo-arc. Its name ``pseudo-arc'' comes from this property because arcs have the same property, in the sense that any subcontinuum (non reduced to a singleton) of an arc is an arc. For more information on pseudo-arc, we refer the interested reader to the second paragraph of~\cite[Chapter XII]{Nadler92} and to~\cite{Bing48,Bing51,Knaster22,Moise48}. Sadly, it is very complicated to get a ``drawing'' of the pseudo-arc due to its complicated crocked structure, see~\cite[Exercise 1.23]{Nadler92}. Following the works of Bing, one can wonder whether the pseudo-arc is typical among arc-like continua and ask whether there is a natural probabilistic construction of the pseudo-arc. 

Let us recall the construction of continua from inverse limits used in \cite{KS20}, see \cite[Section II.2]{Nadler92} for details.  Suppose we are given a sequence 
\begin{displaymath}
  \cdots \xrightarrow[]{f_{3}}X_{3} \xrightarrow[]{f_{2}} X_{2} \xrightarrow[]{f_{1}} X_{1}
\end{displaymath}
where for any $i\geq 1$, the metric space $(X_i,d_{i})$ is compact and $f_i : X_{i+1} \to X_i$ is a continuous surjective function. Then the \emph{inverse limit} of $(\{X_i,f_i\})_{i \geq 1}$ is the subspace of $\prod_{i\geq 1} X_i$ defined by
\begin{equation}
\varprojlim( f_{i},X_{i} : i \geq 1) = \left\{(x_i)_{i \geq 1} \in \prod_{i\geq 1} X_i : f_i(x_{i+1}) = x_i \right\}.
\end{equation}
In the application below $X_{i}$ are compact intervals of $ \mathbb{R}$ and in this case, by \cite[Theorems~2.4 and~12.19]{Nadler92}, the inverse limit is a chainable continuum. In \cite{KS20}, Kiss and Solecki constructed a system as above using two-sided independent Brownian motions $( \mathfrak{B}_{i} : i \geq 1)$. More precisely, they proved that for any interval $J$  of $\RR$ with $0 \in J$ and $J \neq \{0\}$, the following limit exists almost surely 
\begin{equation}
  \mathcal{I}_{i} = \lim_{m \to \infty}  \mathfrak{B}_i\left( \mathfrak{B}_{i+1}\left(\dots \left(  \mathfrak{B}_{i+m} \left(J\right)\right)\dots \right) \right), \label{eq:interval}
\end{equation}
and does not depend on $J$, so that we can consider the random chainable continuum $ \mathcal{C}$ obtained as the inverse limit of the system
\begin{displaymath}
  \cdots \xrightarrow[]{ \mathfrak{B}_{3}} \mathcal{I}_{3} \xrightarrow[]{ \mathfrak{B}_{2}}  \mathcal{I}_{2} \xrightarrow[]{ \mathfrak{B}_{1}}  \mathcal{I}_{1}.
\end{displaymath}

Kiss and Solecki proved \cite[Theorem~2.1]{KS20} that the random chainable continuum $ \mathcal{C}$ is almost surely non-degenerate and indecomposable. This note answers negatively the obvious question the preceding result triggers:

\begin{theorem} \label{thm:notp-a}
Almost surely, the random continuum $ \mathcal{C}$ is \emph{not} hereditary indecomposable (hence is not the pseudo-arc).
\end{theorem}

The proof below could be adapted to prove that a random continuum constructed similarly from a sequence of i.i.d.\ \emph{reflected} Brownian motions is neither a pseudo-arc, answering a question in~\cite[Section~3.1.1]{KS20}. Although almost surely not homeomorphic to the pseudo-arc, the random continuum $ \mathcal{C}$ is interesting in itself and one could ask about its topological property, e.g.~we wonder whether the topology of $ \mathcal{C}$ is almost surely constant and if it is easy to characterise.\\ 

\noindent \textbf{Acknowledgements:} We acknowledge support from the \texttt{ERC 740943} ``GeoBrown'' and \texttt{ANR 16-CE93-0003} ``MALIN''.

\section{Finding good intervals}
In the rest of the article the Brownian motions $ \mathfrak{B}_{i}$ are fixed and we recall the definition of $ \mathcal{I}_{i}$ in \eqref{eq:interval} and of the continuum $ \mathcal{C}$. We will show that Theorem \ref{thm:notp-a} follows from the proposition below stated in terms of images of intervals under the flow of independent Brownian motions whose proof occupy the remaining of the article:
\begin{proposition} \label{prop:UV}
 For any $\varepsilon>0$ small enough, with probability at least $$p_{\varepsilon} = \prod_{i=1}^\infty 1-2 \left( \varepsilon^{{(5/4)}^{i-1}} \right)^{1/8} > 0,$$ there exists two sequences $(U_i)_{i \geq 1}$ and $(V_i)_{i \geq 1}$ of subintervals of $\RR$ such that, for any $i \geq 1$, the five following conditions are satisfied
  \begin{enumerate}
  \item $U_i,V_i \subset  \mathcal{I}_{i}$ where $ \mathcal{I}_{i}$ is defined in \eqref{eq:interval},
  \item $U_i \nsubseteq V_i$ and $V_i \nsubseteq U_i$,
  \item $U_i \cap V_i \neq \emptyset$,
  \item $U_i= \mathfrak{B}_i(U_{i+1})$ and $V_i =  \mathfrak{B}_i(V_{i+1})$,
  \item $|U_i|,|V_i| \leq \varepsilon^{(5/4)^{i-1}}$.
  \end{enumerate}
\end{proposition}

\smallskip
\begin{proof}[Proof of Theorem~\ref{thm:notp-a} given Proposition~\ref{prop:UV}.]
  In the proof, since we are always working with the functions $ \mathfrak{B}_{i}$ we write $\varprojlim( W_{i} : i \geq 1)$ for the inverse limit previously denoted by $\varprojlim(\mathfrak{B}_{i},W_{i}: i \geq 1)$ for any sequence of intervals $W_{1}, W_{2}, ...$ such that $W_{i+1} \xrightarrow{\mathfrak{B}_{i}} W_{i}$. On the event described in the above proposition we have with probability at least $p_\varepsilon > 0$:
  \begin{itemize}
  \item For any $i \geq 1$, $ \mathfrak{B}_i(U_{i+1} \cup V_{i+1}) = U_i \cup V_i$ (point 4) and $U_i \cup V_i \subset  \mathcal{I}_{i}$ (point 1) and $U_i \cup V_i$ is an interval (point 3), so by~Lemma~2.6 of~\cite{Nadler92}, $\varprojlim(U_i \cup V_{i} : i \geq 1)$ is a subcontinuum of $ \mathcal{C}$. 
  \item By Lemma~2.6 of~\cite{Nadler92}, both $\varprojlim(U_i : i \geq 1)$ and $\varprojlim( V_{i} : i \geq 1)$ are also subcontinua of $\varprojlim(U_i \cup V_{i} : i \geq 1)$.
  \item Let $x = (x_i)_{i \geq 1} \in \varprojlim(U_i \cup V_{i} : i \geq 1)$, then
    \begin{itemize}
    \item either, for any $i$, we have $x_i \in U_i \cap V_i$, and so $x \in \varprojlim(U_i  : i \geq 1)$ and $x \in \varprojlim(V_i  : i \geq 1)$,
    \item or there exists $j \geq 1$ such that  $x_j \in U_j$ and  $x_j \notin V_j$, but then by point $4$ we have $x_i \in U_i$ for all $i \geq j$ and so $x \in \varprojlim(U_i  : i \geq 1)$,
    \item or there exists $j \geq 1$ such that  $x_j \notin U_j$ and  $x_j \in V_j$ and similarly we deduce that $x \in \varprojlim(V_i  : i \geq 1)$.
    \end{itemize}
    Hence, $\varprojlim(U_i \cup V_i  : i \geq 1) \subset \varprojlim(U_i  : i \geq 1) \cup \varprojlim(V_i  : i \geq 1)$ and the reverse inclusion is obvious.
  \item $\varprojlim(U_i \cup V_i  : i \geq 1) \ne \varprojlim(U_i  : i \geq 1)$ nor $\varprojlim(U_i \cup V_i  : i \geq 1) \ne \varprojlim(V_i  : i \geq 1)$ by combining point $2$ and point 4.
  \end{itemize}
  All of these points imply that $\varprojlim(U_i \cup V_i  : i \geq 1)$ is a decomposable subcontinuum of  $ \mathcal{C} = \varprojlim( \mathcal{I}_i : i \geq 1)$. That implies that $ \mathcal{C}$ is not a pseudo-arc with probability at least $p_\varepsilon$ for any $\varepsilon > 0$. As $p_\varepsilon \to 1$ when $\varepsilon \to 0$, it is not a pseudo-arc with probability one.
\end{proof}

\subsection{Construction of a decomposable subcontinuum using good shape excursions} \label{sec:decomposable}
Let us now explain the idea behind the construction of the intervals of Proposition \ref{prop:UV}. This relies on the concept of  excursions with a good shape. Imagine that we have a sequence of non trivial intervals $[u_{i},v_{i}] \subset [0,1]$ such that $ \mathfrak{B}_{i}([u_{i+1},v_{i+1}]) = [u_{i}, v_{i}]$ and furthermore that $ \mathfrak{B}_{i}(u_{i+1}) = u_{i}$ and $ \mathfrak{B}_{i}(v_{i+1})= v_{i}$ and $ \mathfrak{B}_{i}(t) \in (u_{i},v_{i})$ for $t \in (u_{i+1}, v_{i+1})$. In words, over the time interval $ [u_{i+1}, v_{i+1}]$, the Brownian motion $ \mathfrak{B}_{i}$ makes an excursion from $u_{i}$ to $v_{i}$. We say that this excursion has a \emph{good shape} if it stays in the pentomino of Figure \ref{fig:pentamino}.

\begin{figure}[!h]
  \begin{center}
    \includegraphics{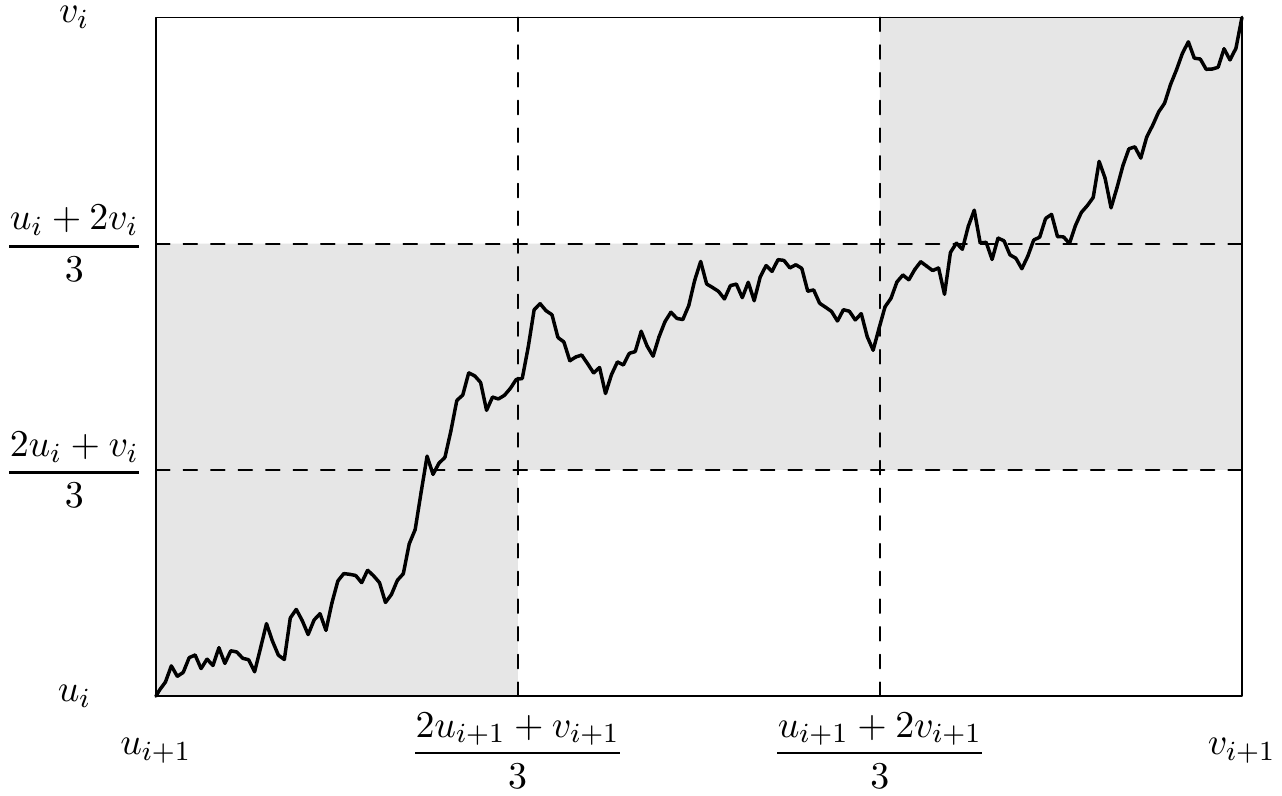}
    \caption{An excursion from $u_{i}$ to $v_{i}$ over the time interval $[u_{i+1}, v_{i+1}]$ has a good shape if it stays in the light grey region.} \label{fig:pentamino}
  \end{center}
\end{figure}
 
If we have such a sequence of intervals and excursions, then one can define a sequence of intervals $U_{i}, V_{i}$ by setting for any $i \geq 1$,
\begin{align*}
  & U_i = \lim_{n \to \infty} \underbrace{\left(  \mathfrak{B}_i \circ \mathfrak{B}_{i+1} \circ \dots \circ \mathfrak{B}_{i+n-1} \right) \left( \left[ u_{i+n},\frac{u_{i+n}+2v_{i+n}}{3} \right] \right)}_{U_{i,n}} \text{ and }\\
  & V_i = \lim_{n \to \infty} \left( \mathfrak{B}_i \circ \mathfrak{B}_{i+1} \circ \dots \circ \mathfrak{B}_{i+n-1} \right) \left( \left[ \frac{2u_{i+n}+v_{i+n}}{3},v_{i+n} \right] \right).
\end{align*}

First, these two limits exist a.s.\ and are closed intervals a.s.\ because they are limits of a sequence of decreasing closed intervals. Indeed, because $ \mathfrak{B}_{i+n}$ performs a good shape excursion from $u_{i+n}$ to $v_{i+n}$ over $[u_{i+n+1},v_{i+n+1}]$ we have
\begin{align*}
   \mathfrak{B}_{i+n}\left( \left[u_{i+n+1},\frac{u_{i+n+1}+2v_{i+n+1}}{3} \right] \right) & \subset \left[u_{i+n},\frac{u_{i+n}+2v_{i+n}}{3} \right], \text{ and so}\\
  U_{i,n+1} & \subset U_{i,n},
\end{align*}
and $U_{i,n}$ are intervals because the BM is continuous a.s. It is then an easy matter to check that the interval constructed above satisfies points 2-4 of Proposition \ref{prop:UV}. Our task is thus to construct the sequence $u_{i},v_{i}$ so that $ \mathfrak{B}_{i}$ performs a good shape excursion from $u_{i}$ to $v_{i}$ over $[u_{i+1}, v_{i+1}]$ and to ensure points $1$ and $5$ of Proposition \ref{prop:UV}. The key idea is to look for these intervals in the vicinity of $0$ because any given small interval close to $0$ has MANY pre-images close to $0$ by a Brownian motion. These many pre-images enable us to select one with a good shape.

\subsection{Pre-images of a small interval by a Brownian motion}
In the following lemma the dependence in $i$ is superfluous but we keep it to make the connection with the preceding discussion easier to understand.
\begin{lemma} \label{lem:uvi}
  Let $a_i$ be any real positive number small enough. Fix $[u_{i},v_{i}] \subset [0, a_i]$. Then with probability at least
  \begin{displaymath}
    1 - 2 a_i^{1/8}
  \end{displaymath}
  we can find $[u_{i+1},v_{i+1}] \subset [0, a_i^{5/4}]$ so that $ \mathfrak{B}_{i}$ performs an excursion with a good shape from $u_{i}$ to $v_{i}$ over the time interval $[u_{i+1}, v_{i+1}]$.
\end{lemma}
\begin{proof}
Fix $0<u_{i}< v_{i}$ and consider the successive excursions $\mathcal{E}_{1}, \mathcal{E}_{2}, ...$ that the Brownian motion $ \mathfrak{B}_{i}$ performs from $u_{i}$ to $v_{i}$  over the respective time intervals $[u_{i+1}^{(1)}, v_{i+1}^{(1)}],[u_{i+1}^{(2)}, v_{i+1}^{(2)}], \cdots$. By the Markov property of Brownian motion and standard argument in excursion theory, these excursions are i.i.d. We claim that 
$$ r = \mathbb{P}( \mathcal{E} \mbox{ has a good shape}) > 0.$$
Indeed, since the law of Brownian motion has full support in the space of continuous functions (with the topology of uniform convergence over all compacts of $ \mathbb{R}_{+}$), the first excursion from $u_{i}$ to $v_{i}$ might be close to any prescribed continuous function and in particular, the probability to have a good shape is strictly positive. See Figure \ref{fig:tube}. 

\begin{figure}[!h]
  \begin{center}
    \includegraphics{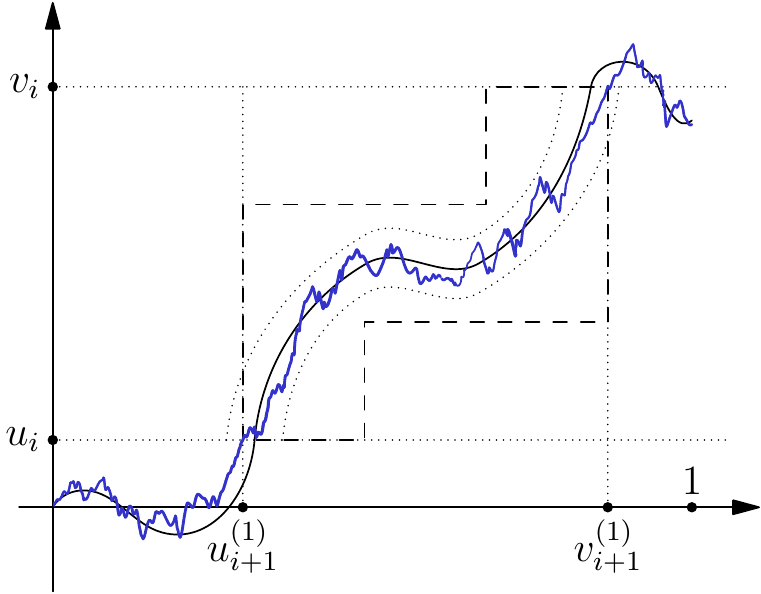}
    \caption{For any given continuous function $f$ starting from $0$ and any $ \varepsilon>0$, the Brownian motion may stay within distance $ \varepsilon>0$ of $f$ up to time $1$ with a positive probability. Choosing $f$ carefully, we deduce that the first excursion from $u_{i}$ to $v_{i}$ has a good shape with positive probability.} \label{fig:tube}
  \end{center}
\end{figure}

Hence, the probability that at least one of the $k$ first excursions has a good shape is at least
\begin{displaymath}
  1 - (1-r)^k.
\end{displaymath}

\begin{figure}[h]
  \begin{center}
    \includegraphics{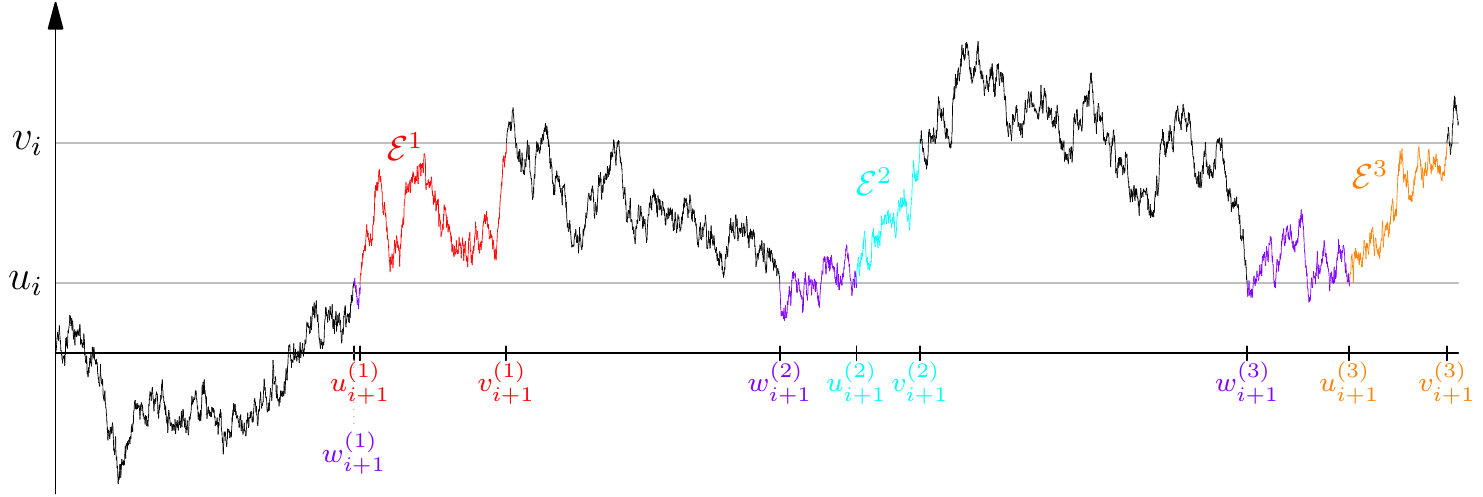}
    \caption{In red, blue and orange, the excursion from $u_i$ to $v_i$ we consider.} \label{fig:excursions1}
  \end{center}
\end{figure}
To control the number of excursions from $u_{i}$ to $v_{i}$ performed up to time $a_{i}$ by $ \mathfrak{B}_{i}$, we introduce the auxiliary stopping times defined by $w_{i+1}^{(1)}= \inf \{t \geq 0: \mathfrak{B}_{i}(t) = u_{i}\}$ and for $k \geq 2$
$$ w_{i+1}^{(k)} = \inf \{t \geq v_{i+1}^{(k-1)} : \mathfrak{B}_{i}(t) = u_{i}\}.$$
Hence $w_{i+1}^{(1)} < v_{i+1}^{(1)} < w_{i+1}^{(2)} < v_{i+1}^{(2)} < \cdots$ are the successive hitting times of $u_{i},v_{i}, u_{i}, v_{i}$ by  $\mathfrak{B}_{i}$, see Figure \ref{fig:excursions1}. For $a \geq 0$, we let $ \mathcal{T}_{a} = \inf\{ t \geq 0 : \mathfrak{B}_{i}(t)=a\}$ the hitting time of $a$ by a standard linear Brownian motion. It is classic (see e.g. \cite[Theorem 2.35]{MP10}) that for $a >0$ we have $ \mathcal{T}_{a} = a^{2} \cdot \mathcal{T}_{1}$ in law where $ \mathcal{T}_{1}$ is distributed according to the L\'evy law 
\begin{displaymath}
  \mathcal{T}_{1} \underset{(d)}{=} \frac{  \mathrm{d}t}{ \sqrt{2 \pi t^{3}}} \mathrm{exp}\left(- \frac{1}{2t}\right) \mathbf{1}_{t>0}.
\end{displaymath}
In our case, applying the strong Markov property at time $w_{i+1}^{(1)} < v_{i+1}^{(1)} < w_{i+1}^{(2)} < v_{i+1}^{(2)} < \cdots$ and using invariance by symmetry we deduce that we have the equalities in distribution 
$$ w_{i+1}^{(1)} \eqd \mathcal{T}_{u_{i}},\quad  v_{i+1}^{(1)} \eqd \mathcal{T}_{u_{i}+ |v_{i}-u_{i}|}, \quad w^{(2)}_{i+1} \eqd  \mathcal{T}_{u_i + 2 |v_i-u_i|}, \cdots , \quad v^{(k)}_{i+1} \eqd  \mathcal{T}_{u_i + (2k-1) |v_i-u_i|},$$ for $k \geq 2$.
Since $ \mathcal{T}_{u_i + (2k-1)|v_i-u_i|} \leq  \mathcal{T}_{2k a_i}$, the probability that the first $k$ excursions of $\mathfrak{B}_i$ occurs before $a_i^{5/4}$ is at least
\begin{displaymath}
  \prob{ \mathcal{T}_{2ka_i} < a_i^{5/4}}=   \prob{ \mathcal{T}_{1} <  \left(\frac{1}{2 k \, a_{i}^{3/8}}\right)} \geq 1 -  \sqrt{\frac{2}{\pi}} 2 k a_i^{3/8} \text{ (for $ka_i^{3/8}$ small enough)}.
\end{displaymath}
Gathering-up the above remarks and taking $k = \lfloor a_i^{-1/4}\rfloor$, we deduce that the probability to do not find an excursion from $u_i$ to $v_i$ with a good shape in $[0,a_i^{5/4}]$ is bounded above by
\begin{displaymath}
  (1-r)^{\lfloor a_i^{-1/4} \rfloor} + 2 \sqrt{\frac{2}{\pi}} \lfloor a_i^{-1/4} \rfloor a_i^{3/8}  \leq 2 a_i^{-1/8} \text{ (for $a_i$ small enough)}. \qedhere
\end{displaymath}
\end{proof}

\section{Proof of Proposition~\ref{prop:UV}} \label{sec:UV}
Let $(\mathfrak{B}_i)_{i \geq 1}$  be a sequence of i.i.d.\ two-sided Brownian motions, and $\varepsilon$ be any real positive number small enough. For any $i \geq 1$, take $a_i = \varepsilon^{(5/4)^{i-1}}$.

Firstly, we put $[u_1,v_1] = [0,\varepsilon] = [0,a_1]$, by Lemma~\ref{lem:uvi}, with probability at least $1 - 2 a_1^{1/8}$, there exists an interval $[u_2,v_2] \subset [0,a_1^{5/4}] = [0,a_2]$ such that $\mathfrak{B}_1$ performs a good shape excursion from $u_1$ to $v_1$ over the time interval $[u_2,v_2]$. Now, we apply Lemma~\ref{lem:uvi} to $[u_2,v_2] \subset [0, a_2]$, etc. At the end, with probability at least
\begin{displaymath}
  \prod_{i=1}^\infty 1 - 2 \left( \varepsilon^{{(5/4)}^{i-1}} \right)^{1/8},
\end{displaymath}
we obtain a sequence of non trivial intervals $([u_i,v_i])_{i \geq 1}$ such that for any $i$, $\mathfrak{B}_i$ makes a good shape excursion from $u_i$ to $v_i$ over  $[u_{i+1},v_{i+1}]$. By Section~\ref{sec:decomposable}, we can then construct two sequences of intervals $U_i$, $V_i$ that satisfy points 2-4 of Proposition~\ref{prop:UV}. Moreover, by construction, $U_i,V_i \subset [u_{i},v_{i}] \subset [0,a_i]$, hence point 5 is also satisfied.\par
Finally, to obtain point 1, just remark that, for any $i,n \geq 1$, $[u_{i+n},v_{i+n}] \subset [0,a_{i+n}] \subset [0,1]$, so
\begin{align*}
  U_i & = \lim_{n \to \infty} \left(  \mathfrak{B}_i \circ \mathfrak{B}_{i+1} \circ \dots \circ \mathfrak{B}_{i+n} \right) \left( \left[ u_{i+n+1},\frac{u_{i+n+1}+2v_{i+n+1}}{3} \right] \right) \\
      & \subset \lim_{n \to \infty} \left(  \mathfrak{B}_i \circ \mathfrak{B}_{i+1} \circ \dots \circ \mathfrak{B}_{i+n} \right) \left( \left[ 0,1  \right] \right) = \mathcal{I}_i \text{ (by~\eqref{eq:interval})}.
\end{align*}
Similarly, $V_i \subset \mathcal{I}_i$. \qed

\bibliographystyle{plain}
\bibliography{../aaa}
\end{document}